\documentclass[11pt]{amsart}
\usepackage{amssymb}
\usepackage{amsmath}
\usepackage{amsthm}
\usepackage{verbatim}
\usepackage[all]{xy}
\usepackage{url}
\usepackage{mathtools}

\usepackage{color}

\numberwithin{equation}{section}

\theoremstyle{plain}
\newtheorem{theorem}[equation]{Theorem}

\newtheorem{lemma}[equation]{Lemma}
\newtheorem{corollary}[equation]{Corollary}

\theoremstyle{definition}

\newtheorem{example}[equation]{Example}
\newtheorem{remark}[equation]{Remark}

\newtheorem{subsec}[equation]{}

\newtheorem{notation*}{Notation}

\DeclareMathOperator{\ch}{char}

\def\C{{\mathbb C}}
\newcommand{\R}{{\mathbb R}}
\newcommand{\A}{{\mathbb A}}

\newcommand{\ov}{\overline}

\newcommand{\Lie}{{\rm Lie}}

\newcommand{\isoto}{\overset{\sim}{\to}}

\newcommand{\labelto}[1]{\xrightarrow{\makebox[1.5em]{\scriptsize ${#1}$}}}

\newcommand{\hs}{\kern 0.8pt}
\newcommand{\hssh}{\kern 1.2pt}
\newcommand{\hshs}{\kern 1.6pt}
\newcommand{\hssss}{\kern 2.0pt}

\newcommand{\hm}{\kern -0.8pt}
\newcommand{\hmm}{\kern -1.2pt}

\newcommand{\kbar}{{\bar k}}

\newcommand{\vphi}{\varphi}
\newcommand{\vk}{\varkappa}

\newcommand{\Hbar}{{\ov H}}

\newcommand{\cc}{\raise 1.7pt \hbox{\Tiny{$\bullet$}}}

\begin{document}

\title[Quotient by a unipotent group]%
{Taking quotient by a unipotent group\\ induces a homotopy equivalence}

\author{Mikhail Borovoi}
\address{%
Borovoi: Raymond and Beverly Sackler School of Mathematical Sciences\\
Tel Aviv University\\
6997801 Tel Aviv\\
Israel}
\email{borovoi@tauex.tau.ac.il}

\author{Andrei Gornitskii}
\address{%
Gornitskii: National Research Nuclear University MEPhI\\
115409 Moscow\\
Russian Federation\\
and
Raymond and Beverly Sackler School of Mathematical Sciences\\
Tel Aviv University\\
6997801 Tel Aviv\\
Israel}
\email{gnomage@mail.ru}

\thanks{This research was supported
by the Israel Science Foundation, grant 870/16.
The second-named author was also supported by RFBR grant No. 20-01-00515}



\date{\today}

\begin{abstract}
Let $U$ be a unipotent group over $\C$ acting
on an irreducible complex algebraic variety $X$.
Assume that there exists a surjective morphism of complex algebraic varieties
$\vphi\colon X\to Y$ whose fibres are orbits of $U$.
We show that if $X$ and $Y$ are smooth
and all orbits of $U$ in $X$ have the same dimension,
then the induced map on $\C$-points $X(\C)\to Y(\C)$ is a homotopy equivalence.
Moreover, if $U$, $X$, $Y$, and $\vphi$ are defined over $\R$,
then the induced map on $\R$-points $X(\R)\to Y(\R)$ is surjective
and induces  homotopy equivalences on connected components.
\end{abstract}

\maketitle

\section{Introduction}
\label{s:Intro}

Let $U$ be a unipotent algebraic group over
the field of complex numbers $\C$,
acting (not necessarily freely) on an algebraic $\C$-variety $X$.
Assume that there exists a surjective morphism
of $\C$-varieties $\vphi\colon X\to Y$
whose fibres are orbits of $U$ in $X$.
In the recent preprint \cite{Maican} Maican proved
that if  $X$ and $Y$ are irreducible smooth $\C$-varieties
and the morphism $\vphi$ is smooth,
then the map on $\C$-points $\vphi_{(\C)}\colon X(\C)\to Y(\C)$
induces isomorphisms on cohomology
\[ H^m(X(\C),A)\labelto\sim H^m(Y(\C),A)\]
for any abelian group $A$ and for all integers $m\ge 0$,
where $X(\C)$ and $Y(\C)$ are regarded with the usual topology.
See \cite[Proposition 4]{Maican}.

In this note we prove a stronger assertion: under the above assumptions,
the map $\vphi_{(\C)}\colon X(\C)\to Y(\C)$ is a homotopy equivalence.
Moreover, assuming that $U,X,Y,$ and $\vphi$
are defined over the field of real numbers $\R$,
we prove that the map $\vphi_{(\R)}\colon X(\R)\to Y(\R)$
is surjective and induces homotopy equivalences on connected components.

\subsection*{\bf Notation}

\begin{itemize}
\item[\cc] $k$ is a field of characteristic 0,
    and $\kbar$ is a fixed algebraic closure of $k$.

\item[\cc] By an algebraic variety over $k$ (in short: a $k$-variety)
    we mean a reduced separated scheme of finite type over $k$.

\item[\cc] For a $k$-variety $X$ and for a field extension $L/k$,
    we denote $X_L=X\times_k L$, the base change of $X$ to $L$,
    and we denote by $X(L)$ the set of $L$-points of $X$.
    We may identify $X(L)=X_L(L)$.

\item[\cc] By  a linear algebraic group over $k$ (in short: a $k$-group)
   we mean  an  affine group scheme of finite type over $k$.
   By Cartier's theorem it is automatically reduced,
   because $\ch k=0$; see Milne \cite[Theorem 3.23]{Milne}.

\item[\cc] By $\Lie(G)$ we denote the Lie algebra of a $k$-group $G$.
\end{itemize}

\section{Main results}
\label{s:main-res}

\begin{theorem}\label{t:hom-space}
Let $Z$ be a right homogeneous space of a unipotent group $U$
over a field $k$ of characteristic 0.
Then $Z$ has a $k$-point. Moreover, $Z$ is isomorphic as a $k$-variety
to the affine space $\A_k^d$, where $d=\dim Z$.
\end{theorem}

We shall discuss the notion of a homogeneous space
and prove Theorem \ref{t:hom-space} in Section \ref{s:homogeneous}.

\begin{subsec}\label{ss:U-X-Y}
Let $U$ be a unipotent $k$-group.
Let $\alpha\colon X\times_k U\to X$ be a right action
of $U$ on a {\em smooth geometrically irreducible} $k$-variety $X$.
Assume that there exists a surjective morphism
onto a {\em smooth geometrically irreducible} $k$-variety $Y$
\[\vphi\colon X\to Y\]
whose fibres over the $\kbar$-points of $Y$
are the orbits of $U(\kbar)$ in $X(\kbar)$.
\end{subsec}

\begin{theorem}\label{t:main}
Let $k$ be either $\C$ or $\R$ (then $\kbar\simeq\C$).
For $U$, $X$, $\alpha$, $Y$, and $\vphi$ as in Subsection \ref{ss:U-X-Y},
assume that the morphism $\vphi_\kbar\colon X_\kbar\to Y_\kbar$ is smooth.
Then:
\begin{enumerate}
\renewcommand{\theenumi}{\roman{enumi}}
\item Each fibre $\vphi^{-1}(y)$ for $y\in Y(k)$ has a $k$-point, that is,
      the map $\vphi_{(k)}\colon X(k)\to Y(k)$ is surjective.
      Moreover, each  fibre $\vphi^{-1}(y)$ is isomorphic as a $k$-variety
      to the affine space $\A^d_k$, where $d=\dim X -\dim Y$.
\item The $C^\infty$-map of $C^\infty$-manifolds $\vphi_{(k)}\colon X(k)\to Y(k)$
      is a locally trivial fibre bundle;
      moreover, it induces  homotopy equivalences on connected components.
\end{enumerate}
\end{theorem}

We shall prove Theorem \ref{t:main} in Section \ref{s:homotopy-eq}.

\begin{corollary}[{Maican \cite[Proposition 4]{Maican}}]
Let $k=\C$ and let $U,\ X,\ \alpha,\  Y,$ and $\vphi$
be as in Theorem \ref{t:main}.
Then for any integer $m\ge 0$ and for any abelian group $A$,
the map  $\vphi_{(\C)}\colon X(\C)\to Y(\C)$ induces an isomorphism
\[ H^m(X(\C), A)\labelto{\sim} H^m(Y(\C), A).\]
\end{corollary}

\begin{proof}
The corollary follows immediately from Theorem \ref{t:main}(ii).
\end{proof}

\begin{remark}
Since $X_\kbar$ and $Y_\kbar$ are smooth,
and the  fibres  $\vphi^{-1}(\vphi(x))=x\cdot U_\kbar$
for $x\in X(\kbar)$ are smooth as well,
the morphism $\vphi_\kbar$ in Theorem \ref{t:main} is smooth
if and only if all fibres have the same dimension $\dim X -\dim Y$.
See Vakil \cite[Theorem 25.2.2 and Exercise 25.2.F (a)]{Vakil}.
\end{remark}

\begin{example}
Let $G$ be a connected linear algebraic $k$-group, 
where $k$ is either $\C$ or $\R$. 
Let $V\subset G$ be a unipotent $k$-subgroup of $G$, 
and let $N\subset V$ be a normal $k$-subgroup of $V$ 
(which is also unipotent).
We construct smooth $k$-varieties  $X=G/N$ and $Y=G/V$.
We have a natural morphism of $k$-varieties
\[\vphi\colon X\to Y,\quad gN\mapsto gV\quad\text{for }g\in G(\kbar).\]
The unipotent $k$-group $U:=V/N$ naturally acts on $X$ on the right by
\[(gN,vN)\mapsto gv N\quad  \text{for }g\in G(\kbar),\ v\in V(\kbar),\] 
and the fibres of $\vphi_\kbar$ are the orbits of $U_\kbar$. 
Moreover, each orbit is a principal homogeneous space of $U_\kbar$.
Since all orbits are of the same dimension  $\dim U$,
the morphism $\vphi_\kbar$ is smooth.
Thus the morphism $\vphi\colon X\to Y$ 
satisfies the hypotheses of Theorem \ref{t:main}.
By this theorem, the $C^\infty$-map $\vphi_{(k)}\colon X(k)\to Y(k)$ 
induces homotopy equivalences on connected components.
\end{example}

\section{Homogeneous space of a unipotent group}
\label{s:homogeneous}

In this section we prove Theorem \ref{t:hom-space}.

\begin{subsec}
\label{ss:hom-space-def}
Let $G$ be a linear algebraic group over $k$.
A {\em right homogeneous space of} $G$
is a $k$-variety $Z$ together with a right action
$\alpha\colon Z\times_k G\to Z$
such that $G(\kbar)$ acts on $Z(\kbar)$ transitively.
We do not assume that $Z$ has a $k$-point.
We shall write just ``a homogeneous space'' for a right homogeneous space.

A  {\em principal homogeneous space (torsor)} of $G$
is a homogeneous space $P$ of $G$ such that
$G(\kbar)$ acts on $P(\kbar)$ {\em simply} transitively, that is,
the stabilizer in $G(\kbar)$ of a $\kbar$-point $ p\in P(\kbar)$ is trivial.

By a {\em morphism} $Z\to Z'$ of homogeneous spaces of $G$
we mean a $G$-equivariant morphism of $k$-varieties.
We say that a homogeneous space $Z$ of $G$ is {\em dominated by a torsor}
if there exists a torsor $P$ of $G$
and a morphism of homogeneous spaces $\psi\colon P\to Z$
(which is surjective on $\kbar$-points
because $P$ and $Z$ are  homogeneous spaces).

Let $z\in Z(\kbar)$ be a $\kbar$-point,
and let $\Hbar\subset G_\kbar$ denote its stabilizer.
The homogeneous space $Z$ defines a {\em $k$-kernel}
($k$-band, $k$-lien) $\vk$ for $\Hbar$
and a cohomology class $\eta(Z)\in H^2(k, \Hbar,\vk)$.
The set $H^2(k, \Hbar,\vk)$ is called the
{\em second nonabelian cohomology set} for $\Hbar$ and $\vk$.
It has a distinguished subset $N^2(k, \Hbar,\vk)$
called the subset of {\em neutral elements}.
The homogeneous space $Z$ is dominated by a torsor if and only if
$\eta(Z)\in N^2(k, \Hbar,\vk)$.
See Springer \cite[Section 1.20]{Springer} for details.
See also \cite[Section 7.7]{Borovoi-Duke} and \cite[Section 1]{FSS}.
\end{subsec}

\begin{lemma}\label{l:unip-dominating}
Let $Z$ be a homogeneous space of a unipotent algebraic group $U$
over a field $k$ of characteristic 0.
Then $Z$ is dominated by a torsor.
\end{lemma}

\begin{proof}
Let $\Hbar\subset U_\kbar$ be the stabilizer of a $\kbar$-point $z\in Z(\kbar)$,
and consider $\eta(Z)\in H^2(k, \Hbar,\vk)$.
Since $\Hbar$ is unipotent and $\ch k=0$,
by Douai's theorem \cite[Theorem IV.1.3]{Douai},
see also \cite[Corollary 4.2]{Borovoi-Duke},
we have $N^2(k, \Hbar,\vk)=H^2(k, \Hbar,\vk)$.
Thus  $\eta(Z)\in N^2(k, \Hbar,\vk)$,
and therefore $Z$ is dominated by a torsor $P$, as required.
\end{proof}

\begin{lemma}\label{l:unip-k-point}
Let $Z$ be as in Lemma \ref{l:unip-dominating}. Then $Z$ has a $k$-point.
\end{lemma}

\begin{proof}
By Lemma \ref{l:unip-dominating} there exists a torsor $P$ of $U$
and a $U$-equivariant $k$-morphism $\psi\colon P\to Z$.
Let $\xi(P)\in H^1(k,U)$ denote the {\em cohomology class of} $P$
in the first Galois cohomology set;
see Serre \cite[Section I.5.2, Proposition 33]{Serre}.
Since $U$ is unipotent group and $\ch k=0$,
by Sansuc's lemma \cite[Lemma 1.13]{Sansuc},
see also \cite[Proposition 3.1]{BDR}, we have $H^1(k,U)=\{1\}$.
It follows that $\xi(P)=1$, and therefore $P$ has a $k$-point $p_0$.
Then $z_0:=\psi(p_0)$ is a $k$-point of $Z$, as required.
\end{proof}

\begin{subsec}{\em Proof of Theorem \ref{t:hom-space}.}
Let $Z$ and $U$ be as in the theorem.
By Lemma \ref{l:unip-k-point}, $Z$ has a $k$-point $z_0$.
Let $H$ denote the stabilizer of $z_0$ in $U$.
Then the map
\[U\to F,\quad u\mapsto z_0\cdot u\]
induces an isomorphism of $k$-varieties $H\backslash U\isoto Z$.
Since $U$ is a unipotent group and $\ch k=0$,
by Proposition (3.2) of Bruce, du Plessis, and Wall \cite{BdPW},
the $k$-variety $H\backslash U$
is isomorphic to the affine space $\Lie(U)/\Lie(H)$,
which, of course, is isomorphic as a $k$-variety to $\A_k^d$,
where $d=\dim\big( \Lie(U) /\Lie(H)\big)=\dim Z$.
This completes the proof of Theorem  \ref{t:hom-space}.\qed
\end{subsec}

\section{Homotopy equivalence}
\label{s:homotopy-eq}

In this section we prove  Theorem \ref{t:main}.

By hypothesis, the fibre $Z_y:=\vphi^{-1}(y)$ for $y\in Y(k)$
is a homogeneous space of $U$.
Therefore, assertion (i) of Theorem  \ref{t:main}
follows immediately from Theorem \ref{t:hom-space}.

We prove assertion (ii) of Theorem \ref{t:main}.
Since $X$ and $Y$ are smooth, the sets of $k$-points $X(k)$ and $Y(k)$
(where $k$ is either $\C$ or $\R$)
are naturally  $C^\infty$-manifolds.
Consider the $C^\infty$-map  $\vphi_{(k)}\colon X(k)\to Y(k)$.
By Theorem  \ref{t:main}(i), the map $\vphi_{(k)}$ is surjective.
Since the morphism $\vphi_\kbar$ is smooth, the map
$\vphi_{(\kbar)}\colon X(\kbar)\to Y(\kbar)$ is a submersion,
that is, for any $x\in X(\kbar)=X_\kbar(\kbar)$, the differential
\[ d_x\hs\vphi\colon T_x(X_\kbar)\to T_{\vphi(x)}(Y_\kbar)\]
is surjective, where $T_x(X_\kbar)$ and $T_{\vphi(x)}(Y_\kbar)$
denote the corresponding tangent spaces.
See Hartshorne \cite[Proposition 10.4]{Hartshorne}.
It follows that the map $\vphi_{(k)}\colon X(k)\to Y(k)$ is a submersion as well.
By Theorem \ref{t:main}(i), each fibre $\vphi_{(k)}^{-1}(y)$ for $y\in Y(k)$
is diffeomorphic to $\A_k^d(k)\simeq k^d$.
Since the $C^\infty$-map $\vphi_{(k)}\colon X(k)\to Y(k)$ is a submersion
with fibres diffeomorphic to $k^d$,
by  Meigniez \cite[Corollary 31]{Meigniez}
it is a locally trivial fibre bundle of $C^\infty$-manifolds.

Let $x_0\in X(k)$, $y_0=\vphi(x_0)\in Y(k)$, $Z=\vphi^{-1}(y_0)\subset X$.
We may identify $Z(k)$ with $k^d$ such that $x_0$ corresponds to $0\in k^d$.
Since $\vphi_{(k)}\colon X(k)\to Y(k)$ is a locally trivial bundle,
we have  a homotopy exact sequence
\[\dots \to\pi_{i+1}(Z(k),x_0)\to \pi_i(X(k),x_0)\labelto{(\vphi_{(k)})_i}
            \pi_i(Y(k),y_0)\to\pi_i(Z(k),x_0)\dots;\]
see, for instance, Arkowitz \cite[Corollary 4.2.19(2)]{Arkowitz}.
Since $\pi_i(Z(k),x_0)=\pi_i(k^d,0)=0$ for all $i$, we conclude
that the  maps $(\vphi_{(k)})_i$ are bijective for all $i\ge 0$.
The topological manifolds $X(k)$ and $Y(k)$
have homotopy types of  CW complexes;
see, for instance, Lundell and Weingram  \cite[Corollary IV.5.7]{LW}.
Hence by Whitehead's theorem
(see, for instance, Hatcher \cite[Theorem 4.5]{Hatcher})
the map $\vphi_{(k)}$ induces a homotopy equivalence
between the connected component of $x_0$ in $X(k)$
and the connected component of $y_0$ in $Y(k)$,
which completes the proof of Theorem \ref{t:main}. \qed

\subsection*{\sc Acknowledgements} 
The authors are grateful to Mario Maican for helpful email correspondence.


\begin{thebibliography}{99}

\bibitem%
{Arkowitz}
Arkowitz, M.
\emph{Introduction to homotopy theory.}
Universitext, Springer-Verlag, New York, 2011.

\bibitem
{Borovoi-Duke}
Borovoi, M. ``Abelianization of the second nonabelian {G}alois
  cohomology.'' \emph{Duke Math. J.} 72, no.~1 (1993), 217--239.

\bibitem
{BDR}
Borovoi, M.,  C. Daw, and J. Ren.
``Conjugation of semisimple subgroups
over real number fields of bounded degree.''
To appear in \emph{Proc. Amer. Math. Soc.,} \url{arXiv:1802.05894[math.GR]},
DOI: 10.1090/proc/14505.

\bibitem%
{BdPW}
 Bruce, J. W., A. A. du Plessis, and  C. T. C. Wall.
 ``Determinacy and unipotency.'' \emph{Invent. Math.} 88 (1987), no. 3, 521--554.

\bibitem
{Douai}
Douai, J.-C. \emph{2-cohomologie galoisienne des groupes semi-simples.}
  Th\`{e}se, Universit\'{e} de Lille I, 1976.

\bibitem
{FSS}
Flicker, Y. Z., C. Scheiderer, and R.~Sujatha. ``Grothendieck's
  theorem on non-abelian {$H^2$} and local-global principles.''
  \emph{J. Amer. Math. Soc.} 11, no.~3 (1998), 731--750.


\bibitem%
{Hartshorne}
Hartshorne, R.
\emph{Algebraic geometry.}
Graduate Texts in Mathematics, No. 52, Springer-Verlag, New York-Heidelberg, 1977.


\bibitem%
{Hatcher}
Hatcher, A.
\emph{Algebraic topology.}  Cambridge University Press, Cambridge, 2002.

\bibitem%
{LW}
Lundell, A. T. and S. Weingram.
\emph{The topology of CW complexes.}
The University Series in Higher Mathematics. Van Nostrand Reinhold Co., New York, 1969.

\bibitem%
{Maican}
Maican, M. ``The effect on topology of the action of a unipotent group.''
 \url{arXiv:2104.00171v1 [math.AG]}.

\bibitem%
{Meigniez}
Meigniez, G. ``Submersions, fibrations and bundles.''
\emph{Trans. Amer. Math. Soc.} 354 (2002), no. 9, 3771--3787.

\bibitem
{Milne}
Milne, J. S.
{\em Algebraic groups.
The theory of group schemes of finite type over a field.}
Cambridge Studies in Advanced Mathematics, 170,
Cambridge University Press, Cambridge, 2017.

\bibitem
{Sansuc}
Sansuc, J.-J. ``Groupe de {B}rauer et arithm\'etique des groupes
  alg\'ebriques lin\'eaires sur un corps de nombres.'' \emph{J. Reine Angew. Math.}
  327 (1981), 12--80.

\bibitem
{Serre}
Serre, J.-P. \emph{Galois cohomology.} Springer-Verlag, Berlin, 1997.

\bibitem
{Springer}
Springer, T. A. ``Nonabelian {$H^{2}$} in {G}alois cohomology.'' In:
  \emph{Algebraic {G}roups and {D}iscontinuous {S}ubgroups ({P}roc. {S}ympos. {P}ure
  {M}ath., {B}oulder, {C}olo., 1965), Proc. Sympos. Pure Math. IX,}  pp.~164--182.
  Amer. Math. Soc., Providence, R.I., 1966.

\bibitem%
{Vakil}
Vakil, R. \emph{The Rising Sea. Foundations of Algebraic Geometry.}
Book in progress, November 18, 2017 draft,
\url{http://math.stanford.edu/~vakil/216blog/FOAGnov1817public.pdf}.

\end{thebibliography}
\end{document}